\documentclass[10pt,a4paper]{article}
\usepackage{latexsym,amsthm,amssymb,amscd,amsmath}
\newcounter{alphthm}
\theoremstyle{plain}
\newtheorem{theorem}{Theorem}[section]
\newtheorem{lemma}[theorem]{Lemma}

\newtheorem{cor}[theorem]{Corollary}
\theoremstyle{definition}
\newtheorem{defn}[theorem]{Definition}

\newtheorem{example}[theorem]{Example}

\newcommand{\be}{\begin{equation}}
\newcommand{\ee}{\end{equation}}
\newcommand{\ben}{\begin{enumerate}}
\newcommand{\een}{\end{enumerate}}

\begin{document}
\title{On $\phi$-Recurrent $(\kappa, \mu)$-Contact Metric Manifolds}
\author{E. Peyghan and A. Tayebi}
\maketitle

\maketitle
\begin{abstract}
In this paper, we show that there is no $\phi$-recurrent Sasakian manifold. Then we prove that the only flat 3-dimensional manifolds are $\phi$-recurrent $(k, \mu)$-contact metric manifolds.
\end{abstract}

\section{Introduction}
Local symmetry is a very strong condition for the class of Sasakian manifolds. Indeed, such spaces must have constant curvature equal to 1 \cite{O}.
Thus Takahashi introduced the notion of a (locally) $\phi$-symmetric space
in the context of Sasakian geometry \cite{T}.
Generalizing the notion of $\phi$-symmetry, De-Shaikh-Biswas  introduced the notion of $\phi$-recurrent Sasakian manifold \cite{DSB}. In  \cite{BBV}, Boeckx-Buecken-Vanhecke  introduced and
studied  the notion of $\phi$-symmetry with several examples. In \cite{Bo}, Boeckx proved that every non-Sasakian $(\kappa, \mu)$-manifold is locally $\phi$-symmetric in the strong sense.

In \cite{JYD}, Jun-Yildiz-De introduced a type of $(\kappa, \mu)$-contact metric manifolds
called $\phi$-recurrent $(\kappa, \mu)$-contact metric manifold which generalizes the notion
of $\phi$-symmetric $(\kappa, \mu)$-contact metric structure of Boeckx. They proved that  three-dimensional locally $\phi$-recurrent $(\kappa, \mu)$-contact metric
manifolds are of constant curvature. They show the existence of $\phi$-recurrent
$(\kappa, \mu)$-manifold by using  an example which is neither locally symmetric nor locally
$\phi$-symmetric.

In this paper, we introduce contact metric manifold, Sasakian manifold, $(\kappa, \mu)$-contact metric manifold and study important properties of these spaces. We show that there exists no $\phi$-recurrent Sasakian manifold. Then we present the example given by Jun-Yildiz-De in \cite{JYD}. In \cite{JYD}, the authors claimed that the 3-dimensional manifold given in this example is a non-Sasakian locally $\phi$-recurrent $(\kappa, \mu)$-contact metric manifold, which is not locally $\phi$-symmetric. We show that this example is not correct and moreover we prove that there is no non-Sasakian locally $\phi$-recurrent $(\kappa, \mu)$-contact metric  manifold with dimension 3, which is not locally $\phi$-symmetric. We also prove that the only flat 3-dimensional manifolds are $\phi$-recurrent $(k, \mu)$-contact  metric manifolds. Finally, we show that there exists no non-flat $(2n+1)$-dimensional $\phi$-recurrent (locally $\phi$-recurrent) contact metric manifold of constant curvature. This assertion show that Theorem 4.1 in \cite{JYD} is not correct.
\section{Contact Metric Manifolds}
We start by collecting some fundamental material about contact metric geometry. We refer to \cite{B}, \cite{BKP} for further details.

A differentiable $(2n + 1)$-dimensional manifold $M^{2n+l}$ is called a {\it{contact manifold}} if it carries a global differential 1-form $\eta$ such that $\eta\wedge(d\eta)^n\neq 0$ everywhere on $M^{2n+1}$. This form $\eta$ is usually called the {\it{contact form}} of $M^{2n+1}$. It is well known that a contact manifold admits an {\it{almost contact metric structure}} $(\phi, \xi, \eta, g)$, i.e., a global vector field $\xi$, which will be called the {\it{characteristic vector field}}, a (1, 1) tensor field $\phi$ and a Riemannian metric $g$ such that
\begin{align}
&(i)\ \eta(\xi)=1,\ \ \ (ii)\ \phi^2=-Id+\eta\otimes\xi,\label{con}\\
&g(\phi X, \phi Y)=g(X, Y)-\eta(X)\eta(Y),\label{con2}
\end{align}
for any vector fields on $M^{2n+1}$. Moreover, $(\phi, \xi, \eta, g)$ can be chosen such that $d\eta(X, Y)=g(X, \phi Y)$ and we then call the structure a {\it contact metric structure} and the manifold $M^{2n+1}$ carrying such a structure is said to be a {\it contact metric manifold}. As a consequence of (\ref{con}) and (\ref{con2}), we have
\[
\phi\xi=0,\ \ \ \eta\circ\phi=0,\ \ \ d\eta(\xi, X)=0.
\]
Denoting by $\pounds$, Lie differentiation, we define the operator $h$ by following
\[
hX:=\frac{1}{2}(\pounds_\xi\phi)X.
\]
The (1, 1) tensor $h$ is self-adjoint and satisfy
\begin{equation}
(i)\ h\xi=0,\ \ \ (ii)\ h\phi=-\phi h,\ \ \ (iii)\ Tr h=Tr h\phi=0.
\end{equation}
Since  the operator $h$ anti-commutes with $\phi$, if $X$ is an eigenvector of $h$ corresponding to the
eigenvalue $\lambda$, then $\phi X$ is also an eigenvector of $h$ corresponding to the eigenvalue
$-\lambda$.

If $\nabla$ is the Riemannian connection of $g$, then
\begin{align}
&\nabla_X\xi=-\phi X-\phi hX,\label{Kill}\\
&\nabla_\xi\phi=0,\\
&g(R(\xi, X)Y, Z)=g((\nabla_X\phi)Y, Z)+g((\nabla_Z\phi h)Y-(\nabla_Y\phi h)Z, X),\\
&2(\nabla_{hX}\phi)Y=-R(\xi, X)Y-\phi R(\xi, X)\phi Y+\phi R(\xi, \phi X)Y-R(\xi, \phi X)\phi Y\nonumber\\
&\ \ \ \  \ \ \ \ \  \ \ \  \ \ \ \ \ \ +2g(X+hX, Y)\xi-2\eta(Y)(X+hX).
\end{align}

A contact structure on $M^{2n+1}$ gives rise to an almost complex structure on
the product $M^{2n+1}\times R$. If this structure is integrable, then the contact metric
manifold is said to be {\it Sasakian}. Equivalently, a contact metric manifold is Sasakian if and only if
\begin{equation}\label{rep}
R(X, Y)\xi=\eta(Y)X-\eta(X)Y.
\end{equation}
Moreover, on a Sasakian manifold the following hold
\begin{equation}
(\nabla_X\phi)Y=g(X, Y)\phi-\eta(Y)X,\ \ \ \ h=0.
\end{equation}
The $(\kappa, \mu)$-nullity distribution of a contact metric manifold $M^{2n+1}(\phi, \xi, \eta, g)$ for the pair $(\kappa, \mu)\in\mathbb{R}^2$ is a distribution
\begin{align*}
N(\kappa, \mu):p\longrightarrow N_p(\kappa, \mu)=\Big\{Z\in T_pM: R(X, &Y)Z=\kappa\big(g(Y, Z)X-g(X, Z)Y\big)\\
&\ \ +\mu\big(g(Y, Z)hX-g(X, Z)hY\big)\Big\}.
\end{align*}
A contact metric manifold $M^{2n+1}(\phi, \xi, \eta, g)$ with $\xi\in N(\kappa, \mu)$ is called {\it $(\kappa, \mu)$-contact manifold}. So, for a $(\kappa, \mu)$-contact manifold, we have
\begin{equation}\label{esi3}
R(X, Y)\xi=\kappa(\eta(Y)X-\eta(X)Y)+\mu(\eta(Y)hX-\eta(X)hY).
\end{equation}
On a $(\kappa, \mu)$-contact metric manifold, $\kappa\leq 1$. If $\kappa=1$, the structure is Sasakian ($h=0$). For $\kappa=\mu=0$,  we have $R(X, Y)\xi=0$. In \cite{B0}, Blair proved the following.
\begin{theorem}\label{Mohim}{\rm (\cite{B0})}
\emph{A contact metric manifold $M^{2n+1}$ satisfying $R(X, Y)\xi=0$ is locally isometric to $E^{n+1}\times S^n(4)$ for $n>1$ and flat for $n=1$, where $E^{n+1}$ is Euclidean
space and  $S^n(4)$ is a sphere of constant curvature 4.}
\end{theorem}
In a $(\kappa, \mu)$-contact metric manifold, the following relations hold (see \cite{B}, \cite{BKP}, \cite{JYD})
\begin{align}
&h^2=(\kappa-1)\phi^2,\ \ \ \kappa\leq 1, \ \  \text{and} \ \kappa=1 \ \text{iff}\  M^{2n+1}\  \text{is Sasakian},\\
&(\nabla_X\eta)(Y)=g(X+hX, \phi Y),\ \ \ S(X, \xi)=2n\kappa\eta(X),\label{rep1}\\
&(\nabla_X\phi)(Y)=g(X+hX, Y)\xi-\eta(Y)(X+hX),\\
&R(\xi, X)Y=\kappa(g(X, Y )\xi-\eta(Y)X)+\mu(g(hX,Y)\xi-\eta(Y)hX),\\
&\eta(R(X, Y)Z)=\kappa(g(Y, Z)\eta(X)-g(X, Z)\eta(Y))\nonumber\\
&\ \  \  \ \ \ \ \ \ \  \ \ \ \ \  \ \ \ \   +\mu(g(hY, Z)\eta(X)-g(hX, Z)\eta(Y))\label{rep2}.
\end{align}
\begin{defn}
\emph{A contact metric manifold $M^{2n+1} (\phi, \xi, \eta, g)$ is said to be locally $\phi$-symmetric if
\begin{equation}
\phi^2((\nabla_WR)(X, Y)Z)=0,\label{a1}
\end{equation}
for all vector fields W, X, Y, Z orthogonal to $\xi$. If (\ref{a1}) holds for all vector fields $W$, $X$, $Y$, $Z$ (not necessarily orthogonal to $\xi$), then we call it $\phi$-symmetric.}
\end{defn}
The notion locally $\phi$-symmetric, was introduced for Sasakian manifolds by Takahashi \cite{T}.
\begin{defn}\label{hasan6}
\emph{A contact metric manifold $M^{2n+1} (\phi, \xi, \eta, g)$ is said to be $\phi$-recurrent if there exists a non-zero 1-form $A$ such that
\begin{equation}\label{rec}
\phi^2((\nabla_WR)(X, Y)Z)=A(W)R(X, Y)Z,
\end{equation}
for all vector fields $X, Y, Z, W$. If the above equation holds for all vector fields W, X, Y, Z orthogonal to $\xi$, then we call it locally $\phi$-recurrent.}
\end{defn}
These notations were introduced for Sasakian manifolds by De-Shaikh-Biswas \cite{DSB} and were introduced for $(\kappa, \mu)$-contact manifolds by Jun-Yildiz-De \cite{JYD}.
\section{Existence of $\phi$-Recurrent Sasakian Manifold}
In \cite{DSB}, De-Shaikh-Biswas introduced the notation Sasakian $\phi$-recurrent contact metric manifold. In this section, we show that there exists no contact metric manifold of this type. Therefore this definition is not well defined.

For a Sasakian manifold,  we have $h=0$ and $\kappa=1$. Then using (\ref{Kill}), (\ref{rep1}) and  (\ref{rep2}),  we have the following.
\begin{lemma}
Let $M^{2n+1}(\phi, \xi, \eta, g)$ be a Sasakian manifold. Then the following relations hold
\begin{align}
&\nabla_X\xi=-\phi X,\label{a2}\\
&(\nabla_X\eta)(Y)=g(X, \phi Y),\label{Sa}\\
&\eta(R(X, Y)Z)=g(Y, Z)\eta(X)-g(X, Z)\eta(Y),\label{a3}\\
&S(X, \xi)=2n\eta(X).\label{Sa2}
\end{align}
\end{lemma}

Here,  we consider the contact metric manifolds with dimension 3. It is known that the Riemannian curvature of a 3-dimensional Riemannian manifold $M$ satisfies in
\begin{align}\label{N50}
R(X, Y)Z&=g(Y, Z)QX-g(X, Z)QY+S(Y, Z)X-S(X, Z)Y\nonumber\\
&\ \ \ \ \  +\frac{r}{2}[g(X, Z)Y-g(Y, Z)X],
\end{align}
where $Q$ is the Ricci operator, that is , $g(QX, Y)=S(X, Y)$ and $r$ is the scalar curvature of $M$.
\begin{theorem}\label{SRE}
There is no $\phi$-recurrent Sasakian manifold with dimension 3.
\end{theorem}
\begin{proof}
Let $M$ be a 3-dimensional $\phi$-recurrent Sasakian manifold. Then the Riemannian curvature of this manifold satisfies in (\ref{N50}). Putting $Z=\xi$ in (\ref{N50}) and using (\ref{Sa2}) and $\eta(\xi)=1$, we obtain
\be
R(X, Y)\xi=(2-\frac{r}{2})[\eta(Y)X-\eta(X)Y]+\eta(Y)QX-\eta(X)QY.\label{a4}
\ee
Then  (\ref{rep}) and (\ref{a4}) give us
\be
(1-\frac{r}{2})[\eta(Y)X-\eta(X)Y]=\eta(X)QY-\eta(Y)QX.\label{a5}
\ee
Setting $Y=\xi$ in (\ref{a5}) and using  (\ref{Sa2}), we get
\begin{equation}\label{Na1}
QX=(\frac{r}{2}-1)X+(3-\frac{r}{2})\eta(X)\xi,
\end{equation}
which gives us
\begin{equation}\label{Na2}
S(X, Y)=g(QX, Y)=(\frac{r}{2}-1)g(X, Y)+(3-\frac{r}{2})\eta(X)\eta(Y).
\end{equation}
By (\ref{N50}), (\ref{Na1}) and (\ref{Na2}), it follows that
\begin{align}
R(X, Y)Z&=(3-\frac{r}{2})[g(Y, Z)\eta(X)\xi-g(X, Z)\eta(Y)\xi+\eta(Y)\eta(Z)X\nonumber\\
&\ \ \ -\eta(X)\eta(Z)Y]+(\frac{r}{2}-2)[g(Y, Z)X-g(X, Z)Y].\label{a6}
\end{align}
From (\ref{a6}) and
\begin{align}\label{FS}
(\nabla_WR)(X, Y)Z&=\nabla_WR(X, Y)Z-R(\nabla_WX, Y)Z\nonumber\\
&\ \ \ -R(X, \nabla_WY)Z-R(X, Y)\nabla_WZ,
\end{align}
we get
\begin{align}
(\nabla_WR)(X, Y)Z&=\frac{dr(W)}{2}[g(Y, Z)X - g(X, Z)Y - g(Y, Z)\eta(X)\xi\nonumber\\
&\ \ + g(X, Z)\eta(Y)\xi - \eta(Y)\eta(Z)X+\eta(X)\eta(Z)Y]\nonumber\\
&\ \ + (3-\frac{r}{2})[g(Y, Z)\eta(X) - g(X, Z)\eta(Y)]\nabla_W\xi\nonumber\\
&\ \ + (3-\frac{r}{2})[\eta(Y)X - \eta(X)Y](\nabla_W\eta)(Z)\nonumber\\
&\ \ + (3-\frac{r}{2})[g(Y, Z)\xi - \eta(Z)Y](\nabla_W\eta)(X)\nonumber\\
&\ \ - (3-\frac{r}{2})[g(X, Z)\xi - \eta(Z)X](\nabla_W\eta)(Y).\label{Im}
\end{align}
Now, let $Y$ be a non-zero vector field orthogonal to $\xi$ and $X=Z=\xi$. Then from (\ref{Im}), we have
\begin{equation}\label{sa}
(\nabla_WR)(\xi, Y)\xi=-2(3-\frac{r}{2})(\nabla_W\eta)(\xi)Y.
\end{equation}
Since $\phi\xi=0$, then using (\ref{Sa}) we obtain
\begin{equation}\label{sa1}
(\nabla_W\eta)(\xi)=g(W, \phi\xi)=0.
\end{equation}
Setting (\ref{sa1}) in (\ref{sa}) yields
\begin{equation}\label{sa2}
(\nabla_WR)(\xi, Y)\xi=0.
\end{equation}
Since $M$ is a $\phi$-recurrent manifold then  there exists a non-zero 1-form $A$ such that satisfies in (\ref{rec}). Thus using (\ref{rec}) and (\ref{sa2}) we deduce that
\begin{equation}\label{sa3}
A(W)R(\xi, Y)\xi=0.
\end{equation}
Since $M$ is Sasakian manifold and $Y$ is a non-zero vector field orthogonal to $\xi$, then we have
\begin{equation}\label{sa4}
R(\xi, Y)\xi=\eta(Y)\xi-\eta(\xi)Y=-Y.
\end{equation}
Setting (\ref{sa4}) in (\ref{sa3}) implies that $A(W)Y=0$, which contradicts with the condition $A(W)\neq 0$.
\end{proof}

\bigskip

\begin{theorem}\label{SRE1}
There is no $\phi$-recurrent Sasakian manifold $M^{2n+1}$ with $n>1$.
\end{theorem}
\begin{proof}
Let $M^{2n+1}$ $(n>1)$, be a $\phi$-recurrent Sasakian manifold. Then using (ii) of (\ref{con}) and (\ref{rec}), we get
\[
-(\nabla_WR)(X, Y)Z+\eta((\nabla_WR)(X, Y)Z)\xi=A(W)R(X, Y)Z,
\]
or
\begin{equation}\label{Bi}
(\nabla_WR)(X, Y)Z=\eta((\nabla_WR)(X, Y)Z)\xi-A(W)R(X, Y)Z,
\end{equation}
where $X, Y, Z, W$ are arbitrary vector fields on $M$ and $A$ is a non-zero 1-form on $M$.
Using Bianchi identity
\[
(\nabla_WR)(X, Y)Z+(\nabla_XR)(Y, W)Z+(\nabla_YR)(W, X)Z=0,
\]
in (\ref{Bi}) implies that
\[
A(W)R(X, Y)Z+A(X)R(Y, W)Z+A(Y)R(W, X)Z=0.
\]
Applying $\eta$ to the above equation yields
\begin{equation}\label{eta}
A(W)\eta(R(X, Y)Z)+A(X)\eta(R(Y, W)Z)+A(Y)\eta(R(W, X)Z)=0.
\end{equation}
By plugging  (\ref{a3}) in (\ref{eta}), it follows that
\begin{align}\label{SR}
&A(W)[g(Y, Z)\eta(X)-g(X, Z)\eta(Y)]+A(X)[g(W, Z)\eta(Y)-g(Z, Y)\eta(W)]\nonumber\\
& \ \ \ \ \ \ \ \ +A(Y)[g(X, Z)\eta(W)-g(W, Z)\eta(X)]=0.
\end{align}
Now, we choose the $\phi$-basis $\big\{e_i, \phi e_i, \xi\big\}_{i=1}^n$ for $M^{2n+1}$ $(n>1)$. By setting $Y=Z=e_i$, $W=e_j$ $(j\neq i)$ and $X=\xi$ in (\ref{SR}),  we obtain
\[
A(e_j)=0.
\]
Since $j$ is arbitrary, then we deduce
\begin{equation}\label{SR1}
A(e_k)=0,\ \ \ \forall\ k=1,\ldots,n.
\end{equation}
Similarly, setting $Y=Z=e_i$, $W=\phi e_j$ and $X=\xi$ in (\ref{SR}) implies
\[
A(\phi e_j)=0.
\]
Thus we deduce
\begin{equation}\label{SR2}
A(\phi e_k)=0,\ \ \ \forall\ k=1,\ldots,n.
\end{equation}
(\ref{rep}) and (\ref{Sa}) give us
\begin{align}
\nonumber(\nabla_WR)(X, Y)\xi&=(\nabla_W\eta)(Y)X-(\nabla_W\eta)(X)Y+R(X, Y)\phi W\\
&=g(W, \phi Y)X-g(W, \phi X)Y+R(X, Y)\phi W.\label{a8}
\end{align}
Putting $X=W=\xi$ in (\ref{a8}) and using $\phi\xi=0$ and $g(\xi, \phi Y)=\eta(\phi Y)=0$, we get
\[
(\nabla_{\xi}R)(\xi, Y)\xi=0.
\]
Thus from (\ref{rec}) we derive that
\[
0=A(\xi)R(\xi, Y)\xi=A(\xi)[\eta(Y)\xi-Y].
\]
If $Y$ is a non-zero vector field orthogonal to $\xi$, then the above equation give us $A(\xi)=0$. Thus by using (\ref{SR}) and (\ref{SR1}),  we deduce that $A=0$ on $M$, which is a contradiction.
\end{proof}
\bigskip

By Theorems \ref{SRE} and \ref{SRE1},  we conclude the following.

\begin{theorem}\label{5.4}
There exists no $\phi$-recurrent Sasakian manifold.
\end{theorem}
\section{$\phi$-Recurrent $(\kappa, \mu)$-Contact Metric Manifolds}
In \cite{JYD}, Jun-Yildiz-De presented the following example (see Section 5 in \cite{JYD}).
\begin{example}\label{Ex}{\rm(\cite{JYD})}
We consider 3-dimensional manifold $M=\{(x, y, z)| x\neq 0\}$, where $(x, y, z)$ are the
standard coordinates in $\mathbb{R}^3$. Let $\{e_1, e_2, e_3\}$ be linearly independent global frame on $M$ given by
\[
e_1:=\frac{2}{x}\frac{\partial}{\partial y},\ \ e_2:=2\frac{\partial}{\partial x}-\frac{4z}{x}\frac{\partial}{\partial y}+xy\frac{\partial}{\partial z},\ \ e_3:=\frac{\partial}{\partial z}.
\]
Let $g$ be the Riemannian metric defined by
\[ g(e_1, e_3)=g(e_2, e_3)=g(e_1, e_2)=0,\ \ g(e_1, e_1)=g(e_2, e_2)=g(e_3, e_3)=0.
\]
Let $\eta$ be the 1-form defined by $\eta(U)=g(U, e_3)$ for any $U\in\chi(M)$. Suppose that $\phi$ be
the (1, 1) tensor field defined by
\[
\phi e_1=e_2, \ \ \phi e_2=-e_1, \ \ \phi e_3=0.
\]
Then using the linearity of $\phi$ and $g$,  we have
\begin{align}
&\eta(e_3)=1, \ \ \ \  \phi^2(U)=-U+\eta(U)e_3,\\
&g(\phi U, \phi W)=g(U, W)-\eta(U)\eta(W) ,
\end{align}
for any $U, W\in\chi(M)$. Moreover
\[
he_1=-e_1, \ \ he_2=e_2, \ \ he_3=0.
\]
Thus for $e_3=\xi$, \ \  $(\phi, \xi, \eta, g)$ defines
a contact metric structure on $M$. Hence we have
\[
[e_1, e_2]=2e_3+\frac{2}{x}e_1,\ \  [e_1, e_3]=0,\ \  [e_2, e_3]=2e_1.
\]
The Riemannian connection $\nabla$ of the metric $g$ is given by
\begin{align}\label{Lev}
2g(\nabla_XY, Z)&=Xg(Y, Z)+Yg(Z, X)-Zg(X, Y)\nonumber\\
&\ \ \ -g(X, [Y, Z])-g(Y, [X, Z])+g(Z, [X, Y]).
\end{align}
Taking $e_3=\xi$ and using the above formula for Riemannian metric $g$, it can
be easily calculated that
\begin{equation} \label{hasan}
\left\{
\begin{array}{cc}
\nabla_{e_1}e_3&\hspace{-.4cm}=0,\ \ \ \ \ \ \ \nabla_{e_2}e_3=2e_1,\ \ \nabla_{e_3}e_3=0, \ \ \nabla_{e_1}e_2=\frac{2}{x}e_1\\
\nabla_{e_2}e_1&\!\!\!\!=-2e_3,\ \ \nabla_{e_2}e_2=0,\ \ \nabla_{e_3}e_2=0,\ \ \nabla_{e_1}e_1=-\frac{2}{x}e_2.
\end{array}
\right.
\end{equation}
By (\ref{hasan}), it is easy to see that $(\phi, \xi, \eta, g)$ is a $(\kappa, \mu)$-contact metric
manifold with $\kappa=-\frac{2}{x}\neq 0$ and $\mu=-\frac{2}{x}\neq 0$.
\end{example}
Now we show that the above example is not correct.\\
Using (\ref{hasan}) we obtain
\begin{align}\label{hasan1}
R(e_1, e_2)e_3&=\nabla_{e_1}\nabla_{e_2}e_3-\nabla_{e_2}\nabla_{e_1}e_3-\nabla_{[e_1, e_2]}e_3\nonumber\\
&=2\nabla_{e_1}e_1-2\nabla_{e_3}e_3-\frac{2}{x}\nabla_{e_1}e_3\nonumber\\
&=-\frac{4}{x}e_2.
\end{align}
But we have
\begin{equation}\label{hasan2}
R(e_1, e_2)e_3=\kappa(\eta(e_2)e_1-\eta(e_1)e_2)+\mu(\eta(e_2)he_1-\eta(e_1)he_2)=0,
\end{equation}
because $\eta(e_1)=g(e_1, e_3)=0$ and $\eta(e_2)=g(e_2, e_3)=0$. Thus (\ref{hasan2}) contradicts (\ref{hasan1}).

\bigskip

In \cite{BKP}, Blair-Koufogiorgos-Papantoniou proved the following.
\begin{lemma}\label{lemma}{\rm (\cite{BKP})}
\emph{Let $M^3$ be a three-dimensional $(\kappa, \mu)$-contact metric manifold and $X$ be a unit eigenvector of $h$, say $hX=\lambda X$, $X$ orthogonal to $\xi$, where $\lambda=\sqrt{1-\kappa}$. Then for $\kappa<1$, we have}
\begin{align*}
[\xi, X]&=(1+\lambda-\frac{\mu}{2})\phi X,\ \ [\phi X, \xi]=(1-\lambda-\frac{\mu}{2})X, \ \ [X, \phi X]=2\xi,\\
\nabla_XX&=\nabla_{\phi X}\phi X=0,\ \ \nabla_X\phi X=(\lambda+1)\xi,\ \ \nabla_{\phi X}X=(\lambda-1)\xi,\\
\nabla_X\xi&=-(1+\lambda)\phi X,\ \ \nabla_\xi X=-\frac{1}{2}\mu\phi X.
\end{align*}
\end{lemma}
\begin{theorem}\label{6.3}
Let $M^3$ be a $(\kappa, \mu)$-contact metric manifold. Then $M$ is $\phi$-recurrent if and only if
$M$ is flat.
\end{theorem}
\begin{proof}
Let $M^3$ be a $(\kappa, \mu)$-contact metric manifold. If $\kappa=1$, then using Theorem \ref{BLAIR} we deduce that $M^3$ is Sasakian. Thus by Theorem \ref{SRE},  we conclude that $M^3$ can not be $\phi$-recurrent. Now let $\kappa<1$ and $X$ be a unit eigenvector of $h$ orthogonal to $\xi$ with corresponding eigenvalue $\lambda=\sqrt{1-\kappa}>0$. Then, according to Lemma \ref{lemma}, there exist three mutually orthonormal vector fields $\xi$, $X$, $\phi X$ such that
\begin{equation}\label{bra}
[X, \phi X]=2\xi,\ \ [\phi X, \xi]=(1-\lambda-\frac{\mu}{2})X,\ \  [\xi, X]=(1+\lambda-\frac{\mu}{2})\phi X,
\end{equation}
where $(\lambda, \mu)\in\mathbb{R}^2$. To simplify in computation, we set
\begin{eqnarray*}
\xi:=e_1,\ \ \  X:=e_2,\ \ \ \phi X:=e_3, \ \  \ c_2:=1-\lambda-\frac{\mu}{2}, \ \  \ c_3:=1+\lambda-\frac{\mu}{2}.
\end{eqnarray*}
Therefore (\ref{bra}) can be written as
\begin{equation}\label{bra1}
[e_2, e_3]=2e_1,\ \ [e_3, e_1]=c_2e_2,\ \  [e_1, e_2]=c_3e_3.
\end{equation}
Since $e_1$, $e_2$ and $e_3$ are orthonormal, then we have $g(e_i, e_j)=\delta_{ij}$. Thus we obtain
\begin{equation}\label{et}
\eta(e_2)=g(e_2, e_1)=0,\ \ \ \eta(e_3)=g(e_3, e_1)=0.
\end{equation}
Using (\ref{bra1}), (\ref{et}) and noting that $\eta(e_1)=\eta(\xi)=1$,  we obtain
\begin{align*}
d\eta(e_3, e_2)&=-d\eta(e_2, e_3)=\frac{1}{2}\eta([e_2, e_3])=1,\\
d\eta(e_i, e_j)&=0,\ \ \ \forall (i, j)\neq(2,3), (3,2).
\end{align*}
Moreover, the condition $ d\eta(e_i, e_j)=g(e_i, \phi e_j)$ gives us
\[
\phi e_1=0, \ \ \ \phi e_2=e_3, \ \ \ \phi e_3=-e_2.
\]
Using (\ref{Lev}), (\ref{bra1}) and  $g(e_i, e_j)=\delta_{ij}$, it follows that
\begin{equation}\label{Con}
\left\{
\begin{array}{cc}
\hspace{-2.3cm}\nabla_{e_1}e_1=0,&\hspace{-2.3cm}\nabla_{e_2}e_2=0,\ \ \nabla_{e_3}e_3=0,\hspace{2cm}\\
\nabla_{e_1}e_2=\frac{1}{2}(c_2+c_3-2)e_3,&\nabla_{e_2}e_1=\frac{1}{2}(c_2-c_3-2)e_3,\\
\hspace{.3cm}\nabla_{e_1}e_3=-\frac{1}{2}(c_2+c_3-2)e_2,&\nabla_{e_3}e_1=\frac{1}{2}(c_2-c_3+2)e_2,\\
\nabla_{e_2}e_3=\frac{1}{2}(c_3-c_2+2)e_1,&\nabla_{e_3}e_2=\frac{1}{2}(c_3-c_2-2)e_1.
\end{array}
\right.
\end{equation}
The Riemannian curvature of $\nabla$ is defined by
\[
R(X, Y)Z=\nabla_X\nabla_YZ-\nabla_Y\nabla_XZ-\nabla_{[X, Y]}Z.
\]
Using (\ref{Con}) and the above equation, one can obtains the following
\begin{align}
R(e_2, e_3)e_2&=\frac{1}{4}[12-4(c_2+c_3)-(c_2-c_3)^2]e_3=(\kappa+\mu)e_3,\label{curva}\\
R(e_2, e_3)e_3&=-\frac{1}{4}[12-4(c_2+c_3)-(c_2-c_3)^2]e_2=-(\kappa+\mu)e_2,\label{curva1}\\
R(e_i, e_j)e_k&=0,\ \ \ \forall i\neq j\neq k\neq i.\label{curva2}
\end{align}
Moreover, since $M^3$ is a $(\kappa, \mu)$-contact metric manifold, then using $he_2=\lambda e_2$ and  $he_3=-\lambda e_3$,  we get
\begin{equation}\label{curva3}
R(e_2, e_1)e_1=(\kappa+\mu\lambda)e_2,\ \ \ R(e_3, e_1)e_1=(\kappa-\mu\lambda)e_3.
\end{equation}
Using (\ref{curva}), (\ref{curva1}), (\ref{curva2}) and (\ref{curva3}) we have
\begin{align}
(\nabla_{e_2}R)(e_2, e_3)e_3&=(\nabla_{e_3}R)(e_2, e_3)e_2=(\nabla_{e_1}R)(e_2, e_3)e_i=0,\ \ i=1, 2, 3,\label{IM1}\\
(\nabla_{e_2}R)(e_2, e_3)e_2&=2(1+\lambda)^2(1-\lambda+\frac{\mu}{2})e_1,\label{IM2}\\
(\nabla_{e_3}R)(e_2, e_3)e_3&=2(\lambda-1)^2(1+\lambda+\frac{\mu}{2})e_1,\label{IM3}\\
(\nabla_{e_2}R)(e_2, e_3)e_1&=-(1+\lambda)[(\kappa+\mu\lambda)e_2+(\kappa+\mu)e_3].\label{Symmetric}
\end{align}
Now, let $M^3$ be the $\phi$-recurrent manifold. Then there exists a non-zero 1-form $A$ on $M$ such that (\ref{rec}) holds for arbitrary vector fields $X$, $Y$, $Z$, $W$ on $M$. Thus using (\ref{rec}), (\ref{curva2}) and (\ref{Symmetric}), it results that
\[
\phi^2\big((\nabla_{e_2}R)(e_2, e_3)e_1\big)=0.
\]
This means that
\[
(1+\lambda)[(\kappa+\mu\lambda)e_2+(\kappa+\mu)e_3]=0.
\]
Since $\lambda>0$, then above equation gives us $\kappa=-\mu$ and $\kappa=-\mu\lambda$. The solution of these equations yield $\kappa=\mu=0$. Thus we have $R(X, Y)\xi=0$ for every vector fields $X$ and $Y$ on $M$. Therefore according to Theorem \ref{Mohim} we conclude that $M^3$ is flat. The converse of the theorem is obvious.
\end{proof}
Now, let $M^3$ be a locally symmetric ($\phi$-symmetric) $(\kappa, \mu)$-contact metric manifold and $\kappa<1$. Then we have $\nabla R=0$ ($\phi^2(\nabla R)=0$). Thus using (\ref{Symmetric}), similar to the proof of the above theorem, we obtain $\kappa=\mu=0$. Therefore we can conclude the following.
\begin{theorem}\label{1}
Let $M^3$ be a non-Sasakian $(\kappa, \mu)$-contact metric manifold. Then $M$ is locally symmetric if and only if $M$ is flat.
\end{theorem}

\bigskip
By the same argument used for the Theorem \ref{1}, we have the following.

\begin{theorem}\label{2}
Let $M^3$ be a non-Sasakian $(\kappa, \mu)$-contact metric manifold. Then $M$ is $\phi$-symmetric if and only if $M$ is flat.
\end{theorem}

\bigskip

Since $\phi e_1=0$, then using (\ref{IM1}), (\ref{IM2}) and (\ref{IM3}) we deduce that
\[
(\nabla_{e_2}R)(e_2, e_3)e_2=(\nabla_{e_2}R)(e_2, e_3)e_3=(\nabla_{e_3}R)(e_2, e_3)e_2=(\nabla_{e_3}R)(e_2, e_3)e_3=0.
\]
Thus for every vector fields $X$, $Y$, $Z$, $W$ orthogonal to $\xi$, it follows that
\[
\phi^2((\nabla_WR)(X, Y)Z)=0.
\]
Therefore we get the following.
\begin{theorem}\label{3}
Every non-Sasakian $(\kappa, \mu)$-contact metric manifold $M^3(\phi, \xi, \eta, g)$ is locally $\phi$-symmetric.
\end{theorem}

\bigskip

By the Theorems \ref{SRE}, \ref{6.3}, \ref{1}, \ref{2},  we conclude the following.
\begin{cor}
There exists no 3-dimensional $\phi$-recurrent $(\kappa, \mu)$-contact metric manifold, which is not locally symmetric (locally $\phi$-symmetric or $\phi$-symmetric).
\end{cor}

Also, from Theorem \ref{3} it results the following.
\begin{cor}
There exists no non-Sasakian locally $\phi$-recurrent $(\kappa, \mu)$-contact metric manifold with dimension 3, which is not locally $\phi$-symmetric.
\end{cor}

\bigskip

In \cite{BKP},  the authors proved the following result.
\begin{theorem}\label{BLAIR} {\rm (\cite{BKP})}
\emph{Let $M^{2n+1}$ be a contact metric manifold with $\xi$ belonging to the $(\kappa, \mu)$-nullity distribution. Then $\kappa\leq 1$. If $\kappa=1$, then $h=0$ and $M^{2n+1}$ is a Sasakian manifold. If $\kappa<1$, then $M^{2n+1}$ admits three mutually orthogonal and integrable distributions $D(0)$, $D(\lambda)$ and $D(-\lambda)$ determined by the eigenspaces of $h$, where $\lambda=\sqrt{1-\kappa}$. Moreover,}
\begin{align}
R(X_{\lambda}, Y_{\lambda})Z_{-\lambda}&=(\kappa-\mu)[g(\phi Y_{\lambda}, Z_{-\lambda})\phi X_{\lambda}-g(\phi X_{\lambda}, Z_{-\lambda})\phi Y_{\lambda}],\label{27}\\
R(X_{-\lambda}, Y_{-\lambda})Z_{\lambda}&=(\kappa-\mu)[g(\phi Y_{-\lambda}, Z_{\lambda})\phi X_{-\lambda}-g(\phi X_{-\lambda}, Z_{\lambda})\phi Y_{-\lambda}],\\
R(X_{\lambda}, Y_{-\lambda})Z_{-\lambda}&=\kappa g(\phi X_{\lambda}, Z_{-\lambda})\phi Y_{-\lambda}+\mu g(\phi X_{\lambda}, Y_{-\lambda})\phi Z_{-\lambda}],\\
R(X_{\lambda}, Y_{-\lambda})Z_{\lambda}&=-\kappa g(\phi Y_{-\lambda}, Z_{\lambda})\phi X_{\lambda}-\mu g(\phi Y_{-\lambda}, X_{\lambda})\phi Z_{\lambda}],\label{30}\\
R(X_{\lambda}, Y_{\lambda})Z_{\lambda}&=[2(1+\lambda)-\mu][g(Y_{\lambda}, Z_{\lambda})X_{\lambda}-g(X_{\lambda}, Z_{\lambda})Y_{\lambda}],\label{31}\\
R(X_{-\lambda}, Y_{-\lambda})Z_{-\lambda}&=[2(1-\lambda)-\mu][g(Y_{-\lambda}, Z_{-\lambda})X_{-\lambda}-g(X_{-\lambda}, Z_{-\lambda})Y_{-\lambda}],\label{310}
\end{align}
\emph{where $X_{\lambda}, Y_{\lambda}, Z_{\lambda}\in D(\lambda)$ and $X_{-\lambda}, Y_{-\lambda}, Z_{-\lambda}\in D(-\lambda)$.}
\end{theorem}
Then they showed the following.
\begin{theorem}\label{THI}{\rm (\cite{BKP})}
\emph{Let $M^{2n+1}$ be a $(\kappa, \mu)$-contact metric manifold with $\kappa<1$. Then the following hold:\\
(i)\ If $X, Y\in D(\lambda)$ (resp. $D(-\lambda)$), then $\nabla_XY\in D(\lambda)$ (resp. $D(-\lambda)$).\\
(ii)\ If $X\in D(\lambda)$, $Y\in D(-\lambda)$, then $\nabla_XY$ (resp. $\nabla_YX$) has no component in
$D(\lambda)$ (resp. $D(-\lambda)$).}
\end{theorem}

Using (\ref{FS}), (\ref{31}) and (i) of Theorem \ref{THI},  we obtain
\begin{align}\label{Im40}
(\nabla_{W_{\lambda}}R)(X_{\lambda}, Y_{\lambda})Z_{\lambda}&=[2(1+\lambda)-\mu][(\nabla_{W_{\lambda}}g(Y_{\lambda}, Z_{\lambda}))X_{\lambda}+g(Y_{\lambda}, Z_{\lambda})\nabla_{W_{\lambda}}X_{\lambda}\nonumber\\
&\ \ -(\nabla_{W_{\lambda}}g(X_{\lambda}, Z_{\lambda}))Y_{\lambda}-g(X_{\lambda}, Z_{\lambda})\nabla_{W_{\lambda}}Y_{\lambda}-g(\nabla_{W_{\lambda}}Y_{\lambda}, Z_{\lambda})X_{\lambda}\nonumber\\
&\ \ +g(X_{\lambda}, Z_{\lambda})\nabla_{W_{\lambda}}Y_{\lambda}-g(Y_{\lambda}, Z_{\lambda})\nabla_{W_{\lambda}}X_{\lambda}+g(\nabla_{W_{\lambda}}X_{\lambda}, Z_{\lambda})Y_{\lambda}\nonumber\\
&\ \ -g(Y_{\lambda}, \nabla_{W_{\lambda}}Z_{\lambda})X_{\lambda}+g(X_{\lambda}, \nabla_{W_{\lambda}}Z_{\lambda})Y_{\lambda}]\nonumber\\
&=[2(1+\lambda)-\mu][(\nabla_{W_{\lambda}}g)(Y_{\lambda}, Z_{\lambda})X_{\lambda}-(\nabla_{W_{\lambda}}g)(X_{\lambda}, Z_{\lambda})Y_{\lambda}]\nonumber\\
&=0.
\end{align}
Similarly, using (\ref{FS}), (\ref{310}) and (i) of Theorem \ref{THI}, it follows that
\[
(\nabla_{W_{-\lambda}}R)(X_{-\lambda}, Y_{-\lambda})Z_{-\lambda}=0.
\]
Therefore we have
\begin{lemma}\label{Esymmetric}
Let $M^{2n+1}$ be a $(\kappa, \mu)$-contact metric manifold. Then $\nabla R$ vanishes on $D(\lambda)$ and $D(-\lambda)$, i.e., we have
\[
(\nabla_{W_{\lambda}}R)(X_{\lambda}, Y_{\lambda})Z_{\lambda}=0, \ \ \ (\nabla_{W_{-\lambda}}R)(X_{-\lambda}, Y_{-\lambda})Z_{-\lambda}=0.
\]
\end{lemma}

\bigskip
Now, we are going to consider the existences of $\phi$-recurrent $(\kappa, \mu)$ contact metric manifold $M^{2n+1}$ with $n>1$.

\begin{theorem}
There is no $\phi$-recurrent $(\kappa, \mu)$ contact metric manifold $M^{2n+1}$ $(n>1)$.
\end{theorem}
\begin{proof}
Let $M^{2n+1}$ $(n>1)$, be a $\phi$-recurrent $(\kappa, \mu)$-contact metric manifold. Then  (\ref{eta}) holds. Setting (\ref{rep2}) in (\ref{eta}), implies that
\begin{align}\label{Mohem}
&A(W)[\kappa\{g(Y, Z)\eta(X)-g(X, Z)\eta(Y)\}+\mu\{g(hY, Z)\eta(X)-g(hX, Z)\eta(Y)\}]\nonumber\\
&+A(X)[\kappa\{g(W, Z)\eta(Y)-g(Z, Y)\eta(W)\}+\mu\{g(hW, Z)\eta(Y)-g(Z, hY)\eta(W)\}]\nonumber\\
&+A(Y)[\kappa\{g(X, Z)\eta(W)-g(W, Z)\eta(X)\}+\mu\{g(hX, Z)\eta(W)-g(hW, Z)\eta(X)\}]\nonumber\\
&=0.
\end{align}
Let $\big\{e_i, \phi e_i, \xi\big\}^n_{i=1}$ be an orthonormal $\phi$-basis with $e_i\in D(\lambda)$. By plugging  $Y=Z=e_i$, $W=e_j (j\neq i)$ and $X=\xi$ in (\ref{Mohem}),  we get
\begin{equation}\label{i}
A(e_j)(\kappa+\mu\lambda)=0.
\end{equation}
Similarly, setting $Y=Z=e_i$, $W=\phi e_l$ and $X=\xi$ in (\ref{Mohem}) we obtain
\begin{equation}\label{ii}
A(\phi e_l)(\kappa+\mu\lambda)=0.
\end{equation}
By using Lemma \ref{Esymmetric} and (\ref{rec}), it follows that
\begin{align*}
0&=\phi^2((\nabla_{e_j}R)(e_i, e_k)e_k)=A(e_j)R(e_i, e_k)e_k=[2(1+\lambda)-\mu]A(e_j)e_i,\ i\neq k,\\
0&=\phi^2((\nabla_{\phi e_l}R)(\phi e_i, \phi e_k)\phi e_k)=A(\phi e_l)R(\phi e_i, \phi e_k)\phi e_k\nonumber\\
&\hspace{4.4cm}=[2(1-\lambda)-\mu]A(\phi e_l)\phi e_i,\ \ i\neq k.
\end{align*}
The above equations give us
\begin{align}
A(e_j)&=0\ \ \text{or}\ \ \mu=2(1+\lambda),\label{v}\\
A(\phi e_l)&=0\ \ \text{or}\ \ \mu=2(1-\lambda).\label{vi}
\end{align}
Now we consider all of cases that would be occur for $A$.\\\\
\textbf{Case 1:} Let $A$ be non-zero on both of $D(\lambda)$ and $D(-\lambda)$. Then there exist $1\leq j, l\leq n$ such that $A(e_j)\neq 0$ and $A(\phi e_l)\neq 0$. In this case, using (\ref{v}), (\ref{vi}) we deduce $\lambda=0$ which is a contradiction.\\\\
\textbf{Case 2:} Let $A$ be non-zero on $D(\lambda)$ and zero on $D(-\lambda)$. Then there exist $1\leq j\leq n$ such that $A(e_j)\neq 0$. In this case, using (\ref{i}) and (\ref{v}) we derive that $\kappa+\mu\lambda=0$ and $\mu=2(1+\lambda)$. A simple substitution yields $\lambda=-1$ which is a contradiction.\\\\
\textbf{Case 3:} Let $A$ be zero on both of $D(\lambda)$ and $D(-\lambda)$. Since $A$ is a non-zero $1$-form, then $A(\xi)\neq0$. Since $g(e_i, \xi)=0$, then $g(\nabla_\xi e_i, \xi)=0$. Thus $\nabla_\xi e_i$ has no component with respect $\xi$. Therefore we can write
\begin{equation}\label{56}
\nabla_\xi e_i=a^r_ie_r+b^r_i\phi e_r.
\end{equation}
Using (\ref{56}), (\ref{27}), (\ref{30}) and (\ref{31}) we obtain
\begin{align}
(\nabla_{_\xi}R)(e_j,e_k)e_s&= [2(1+\lambda)-\mu]\big[(\delta_{ks}b_{j}^r-\delta_{js}b_{k}^r)\phi e_r+a_{j}^se_k-a_{k}^se_j+a_{s}^je_k \nonumber \\
&\ \ -a_{s}^ke_j\big]+\big[\kappa b_{j}^s\phi e_k+\mu b_{j}^k\phi e_s-\kappa b_{k}^s\phi e_j-\mu b_{k}^j\phi e_s+\kappa b_{s}^j\phi e_k\nonumber \\
&\ \ -\kappa b_{s}^k\phi e_j-\mu b_{s}^j\phi e_k+\mu b_{s}^k\phi e_j \big].\label{b1}
\end{align}
By $g(e_j, e_k)=\delta_{jk}$ and (\ref{56}),  we get
\[
a^k_j+a^j_k=0.
\]
Then (\ref{b1}) reduce to the following
\begin{align}
(\nabla_{_\xi}R)(e_j,e_k)e_s&= [2(1+\lambda)-\mu](\delta_{ks}b_{j}^r-\delta_{js}b_{k}^r)\phi e_r+(kb_{j}^s+kb_{s}^j-\mu b_{s}^j)\phi e_k\nonumber\\
&\ \ -(\kappa b_{k}^s+\kappa b_{s}^k-\mu b_{s}^k)\phi e_j+\mu(b_{j}^k-b^j_k)\phi e_s.\label{a9}
\end{align}
Using (\ref{rec}), (\ref{31}) and (\ref{a9}), we deduce that $\mu=2(1+\lambda)$.

Since $\nabla_\xi\xi=0$, then
\begin{equation}\label{eta2}
(\nabla_\xi R)(\xi, Y)\xi=\kappa(\nabla_\xi\eta)(Y)\xi+\mu[h(\nabla_\xi Y)-\nabla_\xi hY].
\end{equation}
Using the first part of (\ref{rep1}) implies that
\[
(\nabla_\xi\eta)(Y)=g(\xi+h\xi, \phi Y)=g(\xi, \phi Y)=\eta(\phi(Y))=0.
\]
By setting the above equation in (\ref{eta2}), one can obtains
\begin{equation}\label{eta3}
(\nabla_\xi R)(\xi, Y)\xi=\mu[h(\nabla_\xi Y)-\nabla_\xi hY].
\end{equation}
On the other hand, (\ref{56}) gives us
\begin{equation}\label{eta4}
h(\nabla_\xi e_i)=\lambda(a^j_ie_j-b^j_i\phi e_j),\ \ \ \nabla_\xi he_i=\lambda(a^j_ie_j+b^j_i\phi e_j).
\end{equation}
Setting $Y=e_i$ in (\ref{eta3}) and using (\ref{eta4}) we get
\[
(\nabla_\xi R)(\xi, e_i)\xi=-2\mu\lambda b^j_i\phi e_j,
\]
which gives us
\[
\phi^2\big((\nabla_\xi R)(\xi, e_i)\xi\big)=2\mu\lambda b^j_i\phi e_j.
\]
Using (\ref{esi3}), (\ref{rec}) and the above equation we have
\[
(\kappa+\mu\lambda)A(\xi)e_i+2\mu\lambda b^j_i\phi e_j=0.
\]
Since $A(\xi)\neq 0$, then from the above equation we deduce  that $\kappa+\mu\lambda=0$. As we see in Case 2, this contradicts $\mu=2(1+\lambda)$.\\\\
\textbf{Case 4.} Let $A$ be non-zero on $D(-\lambda)$ and zero on $D(\lambda)$. Then there exists $1\leq l\leq n$ such that $A(\phi e_l)\neq 0$. In this case, by (\ref{ii}) and (\ref{vi}) we have $\kappa+\mu\lambda=0$ and  $\mu=2(1-\lambda)$ which yield $3\lambda^2-2\lambda-1=0$. This equation and the condition
$\lambda>0$ give us $\lambda=1$. Thus we have $\kappa=\mu=0$.

Now, we compute $(\nabla_{\phi e_l}R)(e_i, e_j)e_k$. Since $g(\nabla_{\phi e_l}e_i, \xi)=-g(e_i, \nabla_{\phi e_l}\xi)=0$, then $\nabla_{\phi e_l}e_i$ has no component with respect $\xi$. Also, from Theorem \ref{THI} we deduce that  $\nabla_{\phi e_l}e_i$ has no component in $D(-\lambda)$. Thus we can write
\begin{equation}\label{fi}
\nabla_{\phi e_l}e_i=\Gamma^r_{li}e_r,
\end{equation}
where $\Gamma^r_{li}=g(\nabla_{\phi e_l}e_i, e_r)$. From (\ref{fi}) and considering $g(e_i, e_r)=\delta_{ir}$,  we obtain
\begin{equation}\label{fi1}
\Gamma^r_{li}=-\Gamma^i_{lr}.
\end{equation}
Using (\ref{31}), (\ref{fi}), (\ref{fi1}) and noting  $\kappa=\mu=0$, it follows that
\begin{equation}\label{fi2}
(\nabla_{\phi e_l}R)(e_i, e_j)e_k=0.
\end{equation}
By using (\ref{rec}), (\ref{31}) and (\ref{fi2}),  we deduce that $A(\phi e_l)(\delta_{jk}e_i-\delta_{ik}e_j)=0$ which gives $A(\phi e_l)=0$ that contradicts the assumption.

According to the above cases, the proof is completes.
\end{proof}
In \cite{B}, Blair proved the following.
\begin{theorem}\label{Blair}{\rm (\cite{B})}
\emph{If a contact metric manifold $M^{2n+1}$ is of constant curvature
$c$ and $n>1$, then $c=1$ and the structure is Sasakian.}
\end{theorem}

Now, we are going to prove the following.
\begin{theorem}\label{Esi}
There is no non-flat $(2n+1)$-dimensional $\phi$-recurrent contact metric manifold of constant curvature.
\end{theorem}
\begin{proof}
Let $M^{2n+1}$ be a non-flat $\phi$-recurrent contact metric manifold. If $n>1$, then according to Theorems \ref{5.4} and  \ref{Blair},  the proof is obvious. Now let $n=1$, i.e., let $M$ be a 3-dimensional non-flat  $\phi$-recurrent contact metric manifold. If $M$ has the constant curvature $c\neq 0$, then we have
\begin{equation}\label{curv}
R(X, Y)Z=c(g(Y, Z)X-g(X, Z)Y).
\end{equation}
Similar to (\ref{Im40}), from (\ref{FS}) and (\ref{curv}) we get
\begin{equation}\label{Imm}
(\nabla_WR)(X, Y)Z=0.
\end{equation}
Putting $Y=Z=\xi$ in (\ref{rec}) and using (\ref{curv}) and (\ref{Imm}) it follows that
\[
cA(W)[X-\eta(X)\xi]=0.
\]
If $X$ is a non zero vector field orthogonal to $\xi$, then the above equation gives us $cA(W)X=0$, which is a contradiction to $c\neq 0$ and $A(W)\neq 0$.
\end{proof}

\bigskip
Now, we are going to consider the existences of  locally $\phi$-recurrent contact metric manifold of constant curvature.

\begin{theorem}
There is no non-flat $(2n+1)$-dimensional locally $\phi$-recurrent contact metric manifold of constant curvature.
\end{theorem}
\begin{proof}
Let $M^{2n+1}$ be a non-flat locally $\phi$-recurrent contact metric manifold. If $M$ has the constant curvature $c\neq 0$, then (\ref{curv}) holds. Therefore, similar to Theorem \ref{Esi}, we deduce that
\begin{equation}\label{IM20}
(\nabla_WR)(X, Y)Z=0.
\end{equation}
Let $\big\{e_i, \phi e_i, \xi\big\}$, $i=1,\ldots, n$, be an orthonormal $\phi$-basis for $M^{2n+1}$. Putting $Y=Z=e_i$ and $X=e_j$ ($j\neq i$) in (\ref{rec}) and using (\ref{curv}) and (\ref{IM20}),  we obtain $cA(W)e_j=0$, which is a contradiction to $c\neq 0$ and $A(W)\neq 0$.
\end{proof}


\bigskip

\noindent
Esmaeil Peyghan\\
Department of Mathematics, Faculty  of Science\\
Arak University\\
Arak 38156-8-8349,  Iran\\
Email: epeyghan@gmail.com
\bigskip

\noindent
Akbar Tayebi\\
Department of Mathematics, Faculty  of Science\\
University of Qom \\
Qom. Iran\\
Email:\ akbar.tayebi@gmail.com

\end{document}